\title{A conjecture of Erd\H{o}s on graph Ramsey numbers}
\author{
Benny Sudakov\thanks{Department of Mathematics,
UCLA,  Los Angeles, CA 90095. Email: {\tt bsudakov@math.ucla.edu}.
Research supported in part by NSF CAREER award DMS-0812005 and by a USA-Israeli BSF grant.}}

\documentclass[11pt]{article}
\usepackage{amsfonts,amssymb,amsmath,latexsym}

\newenvironment{proof}
      {\medskip\noindent{\bf Proof.}\hspace{1mm}}
      {\hfill$\Box$\medskip}

\def\qed{\ifvmode\mbox{ }\else\unskip\fi\hskip 1em plus 10fill$\Box$}

\newtheorem{theorem}{Theorem}[section]

\newtheorem{lemma}[theorem]{Lemma}

\newtheorem{corollary}[theorem]{Corollary}

\newtheorem{definition}[theorem]{Definition}

\begin{document}
\date{}

\maketitle

\begin{abstract}
The Ramsey number $r(G)$ of a graph $G$ is the minimum $N$ such that every
red-blue coloring of the edges of the complete graph on $N$ vertices
contains a monochromatic copy of $G$. Determining or estimating these
numbers is one of the central problems in combinatorics.

One of the oldest results in Ramsey Theory, proved by Erd\H{o}s and Szekeres in
1935, asserts that the Ramsey number of the complete graph with $m$
edges is at most $2^{O(\sqrt{m})}$. Motivated by this estimate 
Erd\H{o}s conjectured, more than a quarter
century ago, that there is an absolute constant $c$ such
that $r(G) \leq 2^{c\sqrt{m}}$ for any graph $G$ with $m$ edges and no isolated vertices. In this short note
we prove this conjecture.
\end{abstract}

\section{Introduction}
Ramsey theory refers to a large body of deep results in mathematics
whose underlying philosophy is captured succinctly by the statement that
``Every large system contains a large well organized subsystem.''
Since the publication of the seminal paper of Ramsey \cite{R30} 
in 1930, this subject has grown into one of the most active areas of research within combinatorics,
overlapping variously with number theory, geometry, analysis, logic and computer science.

Given a graph $G$, the {\it Ramsey number} $r(G)$ is defined to be the smallest
natural number $N$ such that, in any two-coloring of the edges of the complete graph
$K_N$ on $N$ vertices, there exists a monochromatic copy of $G$. 
Existence of $r(G)$ for all graphs follows from Ramsey's theorem 
and determining or estimating these numbers is one of the
central problems in combinatorics (see, e.g., the book \cite{GRS90} for details). 
Probably the most famous question in the field is that of estimating the Ramsey number
$r(K_n)$ of the complete graph on $n$ vertices. A classical result of Erd\H{o}s and
Szekeres~\cite{ES35}, which is a quantitative version of Ramsey's
theorem, implies that $r(K_n) \leq 2^{2n}$ for every positive
integer $n$. Erd\H{o}s~\cite{E47} showed using probabilistic
arguments that $r(K_n)> 2^{n/2}$ for $n> 2$. Over the last sixty years, there have been several
improvements on these bounds (see, e.g., \cite{C08}). However,
despite efforts by various researchers, the constant factors in
the above exponents remain the same. Unsurprisingly then, the field has stretched in different directions and the focus
has turned towards the study of the numbers $r(G)$ for general graphs. 

One such direction that has become fundamental in its own right is that of 
estimating Ramsey numbers for various types of sparse graphs.
In 1975, Burr and Erd\H{o}s \cite{BE75} posed the problem of showing that every
graph $G$ with $n$ vertices and maximum degree $\Delta$ satisfied $r(G) \leq
c(\Delta) n$, where the constant $c(\Delta)$ depends only on $\Delta$. That this is indeed
the case was shown by Chv\'atal, R\"{o}dl, Szemer\'edi and Trotter
\cite{CRST83} in one of the earliest applications of Szemer\'edi's celebrated
regularity lemma \cite{Sz76}. Remarkably, this means that for graphs of fixed
maximum degree the Ramsey number only has a linear dependence on the number of
vertices. However, the use of the regularity lemma only gives a tower-type bound 
on $c(\Delta)$, showing that $c(\Delta)$ is at most an exponential tower of 2s with a 
height that is itself exponential in $\Delta$.

A remarkable new approach to this problem, which avoids the use of any regularity lemma, was developed by
Graham, R\"{o}dl and Ruci\'nski \cite{GrRoRu}. Their proof shows that
that the function $c(\Delta)$ can be taken to be $2^{c \Delta \log^2 \Delta}$ for some 
absolute constant $c$. For bipartite graphs, Graham, R\"{o}dl and Ruci\'nski \cite{GRR01} did even better, showing
that if $G$ is a bipartite graph with $n$ vertices and maximum degree $\Delta$
then $r(G) \leq 2^{c \Delta \log \Delta} n$. They also constructed bipartite graph with $n$ vertices, maximum degree $\Delta$ and
Ramsey number at least $2^{c' \Delta} n$, providing a lower bound on $c(\Delta)$.
Together with Conlon and Fox \cite{C07, FS07, CFS}, we recently further improved these results. Removing the $\log
\Delta$ factor in the exponents, we proved that for general graphs $c(\Delta) \leq 2^{c \Delta \log \Delta}$.
In the bipartite case we achieved an essentially best possible estimate, showing that $r(G) \leq 2^{c \Delta} n$.

Another (somewhat related) problem on Ramsey numbers of general graphs was posed in 1973 by Erd\H{o}s and Graham. Among all graphs with $m$ edges, they wanted to find 
the graph $G$ with maximum Ramsey number. Since, the results we mention so far clearly show that sparse graphs have slowly growing Ramsey numbers, one would 
probably like to make such a $G$ as dense as possible. Indeed, Erd\H{o}s and Graham \cite{EG} conjectured that among all the graphs with $m={n \choose 2}$ edges (and no 
isolated vertices), the complete graph on $n$ vertices has the largest Ramsey number. This conjecture is very difficult and so far there has been no progress on this 
problem. Motivated  by this lack of progress, in the early 80s Erd\H{o}s \cite{E1} (see also \cite{CG98}) asked whether one could at least show that the Ramsey number of 
any graph with $m$ edges is not substantially larger than that of the complete graph with the same size. Since the number of vertices in a complete graph with $m$ edges is a constant 
multiple of $\sqrt{m}$, Erd\H{o}s conjectured that $r(G) \leq  2^{c\sqrt{m}}$ for every graph $G$ with $m$ edges and no isolated vertices. 
Together with Alon and Krivelevich \cite{AlKrSu} we showed that for all graphs with $m$ edges $r(G) \leq  2^{c\sqrt{m}\log m}$
and also proved this conjecture in the special case when $G$ is bipartite. In this paper we  
establish Erd\H{o}s' conjecture in full generality.

\begin{theorem}\label{main}
If $G$ is a graph on $m$ edges without isolated vertices, then $r(G) \leq 2^{250\sqrt{m}}$. 
\end{theorem}
This theorem is clearly best possible up to a constant factor in the exponent, since the result of
Erd\H{o}s (mentioned above) shows that a complete graph with $m$ edges has Ramsey number at least $2^{\sqrt{m/2}}$.

The rest of this short paper is organized as follows. In the next section we present several extensions of 
the well know results which will be our main tools in establishing Theorem \ref{main}.
The proof of this theorem appears in Section 3.  The last section of the paper contains
some concluding remarks and open questions. Throughout the paper, we systematically omit floor and ceiling signs whenever they are not crucial for the sake of clarity of presentation. All
logarithms are in the base $2$. We also do not make any serious attempt to optimize absolute constants in our statements and proofs.

\section{Monochromatic pairs and other tools}
In this section we develop the machinery which we use to establish Theorem \ref{main}. 
We need the following important definition.

\begin{definition}
In an edge-coloring of $K_N$, we call an ordered pair $(X,Y)$ of disjoint subsets of vertices monochromatic if all edges in 
$X \cup Y$ incident to a vertex in $X$ have the same color.
\end{definition}

Our proof has several ingredients, including extensions of two well known results to monochromatic pairs.
The first uses the original argument of Erd\H{o}s and
Szekeres~\cite{ES35} to show how to find such a pair in every $2$-edge-coloring of a complete graph.

\begin{lemma}\label{EESzek}
For all $k$ and $\ell$, every $2$-edge-coloring of $K_N$, contains a monochromatic pair $(X,Y)$ 
with $$|Y| \geq {k+\ell \choose k}^{-1}N-k-\ell$$ which is red and has $|X|=k$ or is  blue with $|X|=\ell$.
\end{lemma}
\begin{proof}
The proof is by induction on $k+\ell$. The base case when $\min(k,\ell)=0$ is trivial. Let $v$ be an arbitrary vertex of $K_N$. 
Then $v$ has either red degree at least $\frac{k}{k+\ell}(N-1)$ or blue degree at least $\frac{\ell}{k+\ell}(N-1)$. 
If $v$ has red degree at least $\frac{k}{k+\ell}(N-1)$, then by induction its set of red neighbors contains a pair  
$(X,Y)$ with $$|Y| \geq {k-1+\ell \choose k-1}^{-1}\frac{k}{k+\ell}(N-1) -(k-1)-\ell \geq 
{k+\ell \choose k}^{-1}N -k-\ell$$ that is monochromatic blue with $|X|=\ell$ (and then we are done)
or monochromatic red with $|X|=k-1$. In the latter case, we can add $v$ to $X$ to obtain a monochromatic red pair
$(X',Y)$ with $X'=X\cup\{v\}$ and $|X'|=k$. A very similar argument, which we omit, can be used to finish the proof in the case when
$v$ has blue degree at least $\frac{\ell}{k+\ell}(N-1)$.
\end{proof}

Although we still do not know how to improve substantially the upper bound for $r(K_n)$,
Erd\H{o}s and Szemer\'edi \cite{ErSzem} showed that this is possible in the case when one color class in the $2$-edge-coloring of $K_N$ is 
very sparse or very dense. The edge density of a graph $G$ is the fraction of pairs of distinct vertices of $G$ that are
edges. Our next lemma extends the result of Erd\H{o}s and Szemer\'edi to monochromatic pairs.

\begin{lemma}\label{EESzem}
Let $0<\epsilon\leq 1/7$ and let $t$ and $N$ be positive integers satisfying $t \geq \epsilon^{-1}$ and $N \geq t\epsilon^{-14\epsilon t}$. Then
any red-blue edge-coloring of $K_N$ in which red has edge density $\epsilon$ 
contains a monochromatic pair $(X,Y)$ with $|X| \geq t$ and $|Y| \geq \epsilon^{14\epsilon t}N$.
\end{lemma}

\begin{proof}
As long as there is a vertex whose red degree is still at least $\epsilon N$ delete it. Since the number of red edges is at most $\epsilon N^2/2$, we have deleted at most $N/2$ vertices. 
Let $S$ denote the set of remaining vertices, so $|S| \geq N/2$ and every vertex in $S$ has red degree at most $\epsilon N$.

If $S$ does not contain a blue clique of size $2t$, let $B$ be a maximum blue clique in $S$. Otherwise, let $B$ be a blue clique in $S$ of size $2t$. 
Delete all vertices of $S \setminus B$ which have at least $3\epsilon |B|$ red neighbors in $B$, and let $S'$ denote the set of remaining vertices. Since every vertex in $B$ has 
red degree at most $\epsilon N$, there are at most $\epsilon N|B|$ red edges from $B$ to $S$, and hence the number of deleted vertices from $S \setminus B$ is at 
most $\frac{\epsilon N|B|}{3\epsilon |B|}= N/3$. Using that $7^{14}t \leq N$, we have $|S'| \geq |S \setminus B|-N/3 \geq N/2-2t-N/3 \geq N/7$. For each 
subset $R \subset B$ of size $3\epsilon |B|$, let $S_R$ denote the set of vertices in $S'$ whose set of red neighbors in $B$ is contained in $R$.

Note that $S'$ is the union of the sets $S_R$, as each vertex in $S'$ has at most $3\epsilon|B|$ red neighbors in $B$.
Since there are ${|B|\choose 3\epsilon |B|}$ such sets, by the pigeonhole principle we have that there is $R$ for which 
$$|S_R| \geq {|B|\choose 3\epsilon |B|}^{-1} |S'| \geq \left(\frac{e}{3\epsilon}\right)^{-3\epsilon|B|} N/7\geq \epsilon^{3\epsilon|B|}N/7\,,$$
where we used the well known fact that ${a \choose b} \leq (ea/b)^b$.

If $|B|=2t$, then let $X=B \setminus R$ and $Y=S_R$. Note that, by definition of $S_R$, all the edges between $B \setminus R$ and $S_R$  are blue. This gives us the monochromatic blue pair $(X,Y)$ with 
$|X| \geq (1-3\epsilon)|B| \geq |B|/2=t$ and 
$|Y| \geq \epsilon^{6\epsilon t}N/7 \geq \epsilon^{14\epsilon t}N$, so we are done. Hence, suppose that $|B|<2t$. 
In this case, note that, there is 
no blue clique of size $|R|+1$ in $S_R$. Indeed, such a blue clique $Q$ in $S_R$ together with $(B \setminus R)$ would form a blue clique in $S$ of size larger than $|B|$, 
contradicting the maximality of $B$. Apply Lemma \ref{EESzek} with $k=t$ and $\ell=7 \epsilon t \geq 3\epsilon |B|+1=|R|+1$ to the coloring restricted to $S_R$.
Since $S_R$ has no blue clique of size $\ell$, it contains a monochromatic red pair 
$(X,Y)$ with $|X|=t$ and 
\begin{eqnarray*}
|Y| &\geq& {t+\ell \choose \ell}^{-1}|S_R|-t-\ell \geq
\left(\frac{\ell}{e(t+\ell)}\right)^{\ell}|S_R|-2t \geq
(1.2\epsilon)^{7\epsilon t}|S_R|-2t\\
&\geq& \frac{1.2^7}{7}\epsilon^{7\epsilon t} \epsilon^{3\epsilon|B|}N-2t \geq 
\frac{1}{2}\epsilon^{13\epsilon t}N-2\epsilon^{14\epsilon t}N \geq \left(\frac{1}{2}-2\epsilon \right) \epsilon^{13\epsilon t}N \geq \epsilon^{14\epsilon t}N, 
\end{eqnarray*}
where we used that $\epsilon \leq 1/7,~ \frac{\ell}{e(t+\ell)}=\frac{7\epsilon}{e(1+7\epsilon)} \geq 1.2\epsilon,~ 1.2^7 \geq 3.5 ,~ t \leq \epsilon^{-1},~ 2t \leq 2\epsilon^{14\epsilon t}N \leq 
2\epsilon \cdot \epsilon^{13\epsilon t}N$.
This completes the proof of the lemma. 
\end{proof}

Finally, we need some tools developed by
Graham, R\"odl  and Ruci\'nski \cite{GrRoRu} to study the Ramsey numbers of sparse graphs
(see also \cite{FS} for simpler proofs and generalizations). We start with some
notation. Let $H$ be a graph with vertex set $V$ and let $U$ be a
subset of $V$. Then we denote by $H[U]$ the subgraph of $H$ induced by $U$ and by
$e(U)$ its number of edges. 
The {\em edge density} $d(U)$ of $U$ is defined by
$$d(U)=\frac{e(U)}{{|U| \choose 2}}.$$ 
Similarly, if $X$ and $Y$ are two disjoint subsets of $V$, then
$e(X,Y)$ is the number of edges of $H$ adjacent to exactly one vertex from
$X$ and one from $Y$ and the density of the pair $(X,Y)$ is defined by
$$d(X,Y)=\frac{e(X,Y)}{|X||Y|}.$$
We say that $H$ is {\em $(\rho, \epsilon)$-sparse} if 
there is a pair of disjoint subsets  $X, Y \subset V$ with $|X|=|Y| \geq \rho|V|$ and 
$d(X,Y)\leq \epsilon$. The following lemma was proved in \cite{GrRoRu} (see Lemma 2). It shows that
if the density between every two sufficiently small disjoint subsets of $H$ is at least $\epsilon$, then $H$ contains 
every bounded degree graph $G$ of order proportional to $V(H)$.

\begin{lemma}
\label{l23}
Let $\Delta$ and $n$ be two integers, $\epsilon \leq 1/2$ and $\rho=\epsilon^{\Delta}/(\Delta+1)$.
Let also $G$ be a graph on $n$ vertices with maximum degree at most $\Delta$. If $H$ is a graph of order at least
$(\Delta+1)\epsilon^{-\Delta}n$ which contains no copy of $G$ then $H$ 
is $(\rho, \epsilon)$-sparse.
\end{lemma}

Note that if the  graph $H$ contains no copy of graph $G$ then after finding one sparse pair $(X,Y)$ one can apply the last lemma again to the subgraphs of $H$ induced by sets $X$ and $Y$.
By doing this recursively, it was proved in \cite{GrRoRu} and \cite{FS} that $H$ contains a sparse subset. We use the following statement from \cite{FS}
(see Corollary 3.4).

\begin{lemma}
\label{l24}
Let $0 \leq \epsilon, \rho \leq 1$, $h=\log (2/\epsilon)$ and let $H=(V,E)$ be a graph such that for every subset $U$ of $H$ of size
at least $(\rho/2)^{h-1}|V|$ the induced subgraph $H[U]$ is $(\rho, \epsilon/8)$-sparse. Then $H$ contains a subset $S, |S| \geq 2\rho^h|V|$ 
with edge density $d(S) \leq \epsilon$.
\end{lemma}

Combining these two lemmas we obtain the last ingredient, which we need for the proof of our main result.

\begin{corollary} 
\label{GRR}
Let $G$ be a graph with $n$ vertices, maximum degree $\Delta$ and let $\epsilon \leq 1/8$. 
If $H$ has $N \geq \epsilon^{4\Delta\log \epsilon}n$ vertices and does not contain a copy of $G$, then it has a 
subset $S$ of size $|S| \geq \epsilon^{-4\Delta \log \epsilon}N$ with edge density $d(S) \leq \epsilon$.
\end{corollary}
\begin{proof}
Let $\rho=(\epsilon/8)^{\Delta}/(\Delta+1)$, $h=\log (2/\epsilon)$ and let $U$ be a subset of $H$ of size 
$(\rho/2)^{h-1}N$. Using that $\epsilon \leq 1/8$  it is rather straightforward to check that $n \leq (\rho/2)^h N$ and therefore 
$|U|  \geq \rho^{-1} n =(\Delta+1) (\epsilon/8)^{-\Delta}n$. Moreover the induced subgraph $H[U]$ contains no copy of $G$, 
and therefore satisfies the conditions of Lemma \ref{l23} (with $\epsilon/8$ instead of $\epsilon$). Therefore $H[U]$ is 
$(\rho, \epsilon/8)$-sparse. Thus we can apply Lemma \ref{l24} to $H$ and find a subset $S, |S| \geq 2\rho^h N $
with density $d(S) \leq \epsilon$. Since $2\rho^h \geq \epsilon^{-4\Delta \log \epsilon}$, this completes the proof.
\end{proof}

\section{Proof of the main result}
We start by describing the idea of the proof. Suppose we have a red-blue edge-coloring of $K_N$ without any monochromatic copy of a certain graph $G$ with 
$m$ edges. Using Lemma \ref{EESzek}, we find a monochromatic pair $(X,Y)$ (say in red), where the  size of $X$ is of order $\sqrt{m}$.
Split the graph $G$ into two parts $A$ and $G'=G\setminus A$, where $A$ is a set of $|X|$ vertices of the largest degree in $G$. Then
the graph $G'$ will have maximum degree bounded by $2m/|X|=O(\sqrt{m})$. Embed $A$ into $X$ and try to find a red copy of $G'$ in $Y$. If $Y$ has no red copy of $G'$, 
use Corollary \ref{GRR} to conclude that it has a relatively large subset $S$ with few red edges. Now we can apply Lemma \ref{EESzem} to $S$ to 
find a new monochromatic pair $(X',Y')$ in blue with the 
following 
property. 
The size of $X'$ will be considerably larger than the size of $X$. On the other hand, the size of 
$Y'$ will not decrease substantially compared with the 
size of $Y$. The proof will follow by repeated application of this procedure, since at some point the size of the 
monochromatic clique $X'$ will be larger than the number of vertices in $G$. The following key lemma gives a precise formulation of the amplification step. 

\begin{lemma}\label{pairslemma} 
Let $G$ be a graph with $m$ edges and without isolated vertices and suppose $27 \leq \alpha \leq \frac{1}{8}\log^3 m$.
If a red-blue edge-coloring of a complete graph on $N$ vertices has no monochromatic copy 
of $G$ and contains a monochromatic pair $(X,Y)$ with $|X| \geq \alpha \sqrt{m}$ and $|Y| \geq 2^{125\alpha^{-1/3}\sqrt{m}}$, 
then it also contains a monochromatic pair $(X',Y')$ with $|X'| \geq 2^{2\alpha^{1/3}}\sqrt{m}$ 
and $|Y'| \geq 2^{-120\alpha^{-1/3}\sqrt{m}}|Y|$.
\end{lemma}
\begin{proof}
Without loss of generality, assume that the color of the monochromatic pair $(X,Y)$ is red. 
Let $G'$ be the induced subgraph of $G$ formed by deleting the $|X|$ vertices of $G$ of 
largest degree. As $G$ has $m$ edges, it has $n \leq 2m$ vertices and the maximum degree of $G'$ satisfies $\Delta(G') \leq \frac{2m}{|X|}=\frac{2}{\alpha}\sqrt{m}$. 
The coloring restricted to $Y$ contains no monochromatic red copy of $G'$ as, otherwise, together with $X$, 
we would get a monochromatic copy of $G$. Let $\epsilon=2^{-3\alpha^{1/3}}$ and let $t=2^{2\alpha^{1/3}}\sqrt{m}$. 
Since $2^{\alpha^{1/3}}\leq \sqrt{m}$ we have that  $t \geq \epsilon^{-1}$.
Also note that, since $27 \leq \alpha \leq \frac{1}{8}\log^3 m$, we have that $42\alpha^{1/3}2^{-\alpha^{1/3}}\leq 48\alpha^{-1/3}$ and
$2^{5\alpha^{-1/3}\sqrt{m}} \geq 2^{10\sqrt{m}/\log m} \geq m^{3/2} \geq 2^{2\alpha^{1/3}}\sqrt{m}=t$.
Applying Corollary \ref{GRR} to the red graph restricted to $Y$, we find a subset $S \subset Y$ with 
\begin{eqnarray*}
|S| &\geq& \epsilon^{-4\Delta(G') \log \epsilon}|Y| \geq 
\Big(2^{-3\alpha^{1/3}}\Big)^{-4 \cdot (2\alpha^{-1}\sqrt{m}) \cdot(-3 \alpha^{1/3})}|Y| = 2^{-72\alpha^{-1/3}\sqrt{m}}|Y|\\
&\geq& 
2^{53\alpha^{-1/3}\sqrt{m}}>m^3 \geq n
\end{eqnarray*}
such that the red density in $S$ is at most $\epsilon$. 
Then,  the size of $S$ satisfies 
$$|S| \geq  2^{53\alpha^{-1/3}\sqrt{m}} \geq 2^{5\alpha^{-1/3}\sqrt{m}}\cdot 2^{48\alpha^{-1/3}\sqrt{m}}
\geq t\, 2^{42\alpha^{1/3}2^{-\alpha^{1/3}}\sqrt{m}}=t \epsilon^{-14\epsilon t},$$
and we can apply Lemma \ref{EESzem} to $S$. By this lemma, $S$ contains a monochromatic pair 
$(X',Y')$ with $|X'|=t$ and $|Y'| \geq \epsilon^{14\epsilon t}|S|$. To complete the proof, recall that $|S| \geq 2^{-72\alpha^{-1/3}\sqrt{m}}|Y|$, and therefore  
$$|Y'| \geq \epsilon^{14\epsilon t}|S| \geq 2^{-48\alpha^{-1/3}\sqrt{m}}|S| \geq 2^{-120\alpha^{-1/3}\sqrt{m}}|Y|.$$
\end{proof}

\vspace{0.1cm}
\noindent {\bf Proof of Theorem \ref{main}.} 
Let $G$ be a graph with $m$ edges and without isolated vertices. Note that $G$ has at most $2m$ vertices. 
Suppose for contradiction that there is a red-blue edge-coloring of $K_N$ with $N=2^{250\sqrt{m}}$ which contains no
monochromatic copy of $G$. Since, as was mentioned in the introduction, $r(K_{2m}) \leq 2^{4m}$ we have that
$m \geq 60^2$. Applying Lemma \ref{EESzek} with $k=\ell=27\sqrt{m}$, we find  a monochromatic pair $(X_1,Y_1)$ with $|X_1| \geq 27\sqrt{m}$ and 
$$|Y_1| \geq {k+\ell \choose k}^{-1}N-k-\ell \geq 4^{-27\sqrt{m}}N= 2^{196\sqrt{m}}.$$ 
Define $\alpha_1=27$ and $\alpha_{i+1}=2^{2\alpha_i^{1/3}}$. An easy computation shows that $\alpha_{i+1}\geq (4/3)^3 \alpha_i$ for all $i$ and therefore
$\alpha_i^{-1/3} \leq 3^{-1}(3/4)^{i-1}$. In particular this implies that
$$\sum_{j =1}^i \alpha_j^{-1/3} \leq \frac{1}{3}\sum_{j=0}^{i} (3/4)^{-j}=
\frac{1}{3}\sum_{j \geq 0} (3/4)^{-j}- \frac{1}{3}\sum_{j \geq i+1} (3/4)^{-j} \leq 4/3-4\alpha_{i+1}^{-1/3}.$$ 

Since the red-blue edge-coloring has no monochromatic copy of $G$, we can repeatedly apply Lemma \ref{pairslemma}. 
After $i$ iterations, we have a monochromatic pair $(X_{i+1},Y_{i+1})$ with $|X_{i+1}| \geq \alpha_{i+1}\sqrt{m}$ and
\begin{eqnarray*} 
|Y_{i+1}| &\geq& 2^{-120\alpha_i^{-1/3}\sqrt{m}}|Y_i| \geq 2^{-120\sqrt{m} \sum_{j= 1}^i \alpha_j^{-1/3}}|Y_1| \geq
2^{196\sqrt{m}} 2^{-120\sqrt{m}(4/3-4\alpha_{i+1}^{-1/3})}\\
&\geq& 2^{(36+480\alpha_{i+1}^{-1/3})\sqrt{m}}.
\end{eqnarray*}
Hence we can continue iterations until the first index $i$ such that $\alpha_{i} \geq \frac{1}{8}\log^3 m$. Then 
for $\alpha = (\log m/2)^3 \geq 5^3$ we have a monochromatic pair $X=X_i, |X| \geq \alpha \sqrt{m}$ and $Y=Y_i, |Y| \geq 2^{36\sqrt{m}} \geq
2^{125\alpha^{-1/3}\sqrt{m}}$. Then applying 
Lemma \ref{pairslemma} one more time we obtain a monochromatic pair $(X',Y')$ with $|X'| \geq 2^{2\alpha^{1/3}}\sqrt{m}=m^{3/2} \geq 2m$ 
and we can embed $G$ (which has at most $2m$ vertices) into the monochromatic clique $X'$, a contradiction.
\hfill $\Box$ 

\section{Concluding remarks}

A graph is {\it $d$-degenerate} if every induced subgraph of it has a vertex
of degree at most $d$. Notice that graphs with maximum degree $d$
are $d$-degenerate. This notion nicely captures the concept of
sparse graphs as every $t$-vertex subgraph of a $d$-degenerate graph
has at most $td$ edges. (Indeed, remove from the subgraph a vertex
of minimum degree, and repeat this process in the remaining
subgraph.) Burr and Erd\H{o}s \cite{BE75} conjectured that, for each
positive integer $d$, there is a constant $c(d)$ such that $r(H)
\leq c(d)n$ for every $d$-degenerate graph $H$ on $n$ vertices. This
well-known and difficult conjecture is a substantial generalization
of the results on Ramsey numbers of bounded-degree
graphs (mentioned in introduction) and progress on this problem was made only recently.

Improving an earlier polynomial bound of \cite{KoRo1}, we obtained, together with Kostochka,
the first nearly linear bound on the Ramsey numbers of $d$-degenerate
graphs. In \cite{KoSu} we proved that such graphs satisfy $r(H)
\leq c(d)n^{1+\epsilon}$ for any fixed $\epsilon>0$. The best current estimate, showing that 
$r(H) \leq 2^{c(d) \sqrt{\log n}} n$, is due to Fox and Sudakov \cite{FS09}.
In spite of this progress, the Burr-Erd\H{o}s conjecture is still open even for the very special case of $2$-degenerate
graphs. However, it seems plausible that $r(H) \leq 2^{c d}n$ (for some constant $c$) holds in general
for every $d$-degenerate graph $H$ on $n$ vertices. Such an estimate would be a far reaching generalization of the  
results about Ramsey numbers of bounded degree graphs and also of Theorem \ref{main}. Indeed, it is easy to check that
every graph with $m$ edges is $\sqrt{2m}$-degenerate.

Finally, we would like to stress that the proofs given in this paper are highly
specific to the 2-color case. 
The $k$-color Ramsey number $r_k(G)$ is the least positive integer $N$ such
that every $k$-coloring of the edges of a complete graph
$K_N$, contains a monochromatic copy of the graph $G$ in one of the colors.
It would be interesting to understand, for $k \geq 3$, the order of magnitude of the $k$-color Ramsey number of a 
graph with $m$ edges.
Also for bounded degree graphs the best results that are known in the $k$-color
case are much worse than for $2$ colors. For example (see \cite{FS07}), for a 
graph $G$ on $n$ vertices with maximum degree $\Delta$, 
we only know the bound  $r_k(G) \leq 2^{c(k) \Delta^2} n$. Improving it to $r_k(G) \leq 2^{c(k) \Delta^{1+o(1)}} n$ would be of considerable
interest.

\vspace{0.3cm}
\noindent
{\bf Acknowledgment.} I would especially like to thank Jacob Fox for valuable comments and 
for writing  a draft of the proof of the main theorem. I also had many stimulating discussions with David Conlon and want to thank
him for reading very carefully an early version of this manuscript.

\end{document}